\documentclass[a4paper]{amsart}
\usepackage[T1]{fontenc}
\usepackage[utf8]{inputenc}

\usepackage{amsmath}
\usepackage{amssymb}
\usepackage{amsthm}

\usepackage{mathpazo,courier}
\usepackage[scaled]{helvet}

\usepackage[numbers,sort&compress]{natbib}

\usepackage{enumerate}

\usepackage{tikz-cd} %\usetikzlibrary{cd}
\usepackage{hyperref}

\newcommand{\Z}{\mathbb{Z}}

\newcommand{\R}{\mathbb{R}}
\newcommand{\C}{\mathbb{C}}

\newcommand{\A}{\mathbb{A}}

\newcommand{\cI}{\mathcal{I}}
\newcommand{\cC}{\mathcal{C}}
\newcommand{\cB}{\mathcal{B}}

\newcommand{\p}{\mathfrak{p}}
\newcommand{\q}{\mathfrak{q}}

\newcommand\ph\varphi
\newcommand\ps\psi
\newcommand\ep\varepsilon
\newcommand\rh\varrho
\newcommand\al\alpha
\newcommand\be\beta
\newcommand\ga\gamma
\newcommand\om\omega
\newcommand\ta\tau
\renewcommand\th\vartheta
\newcommand\de\delta
\newcommand\ze\zeta
\newcommand\ch\chi
\newcommand\et\eta
\newcommand\io\iota
\newcommand\la\lambda
\newcommand\si\sigma

\newcommand\Ga\Gamma
\newcommand\De\Delta
\newcommand\Th\Theta
\newcommand\La\Lambda
\newcommand\Si\Sigma
\newcommand\Ph\Phi
\newcommand\Ps\Psi
\newcommand\Om\Omega

\DeclareMathOperator\Quot{Quot}

\DeclareMathOperator\Hom{Hom}

\DeclareMathOperator\Cl{Cl}
\DeclareMathOperator\Tr{Tr}
\DeclareMathOperator\Norm{N}

\DeclareMathOperator\Spec{Spec}
\DeclareMathOperator\Sper{Sper}

\DeclareMathOperator\Sym{Sym}
\DeclareMathOperator\Mat{Mat}

\newtheorem{thm}{Theorem}[section]
\newtheorem{mthm}{Theorem}
\newtheorem{prop}[thm]{Proposition}
\newtheorem{lem}[thm]{Lemma}
\newtheorem{cor}[thm]{Corollary}
\theoremstyle{definition}
\newtheorem{rem}[thm]{Remark}

\renewenvironment{proof}[1][\unskip]{\par\noindent {\em Proof #1: }}{{\qed\bigskip}}

\title[Characteristic Polynomials of Symmetric Matrices]{\normalsize Characteristic Polynomials of Symmetric Matrices over the Univariate Polynomial Ring}
\date{\small\today}
\author{\small Christoph Hanselka}\thanks{The result of the present paper is part of the authors thesis \citep{Hanselka15} written  under the supervision of Prof. Markus Schweighofer at the University of Konstanz.}
\address{The University of Auckland,Department of Mathematics, Private Bag 92019, Auckland 1142, New Zealand}
\email{c.hanselka@auckland.ac.nz}
\begin{document}
\begin{abstract}
	Viewing a bivariate polynomial $f\in\R[x,t]$ as a family of univariate polynomials in $t$ parametrized by real numbers $x$, we call $f$ \emph{real rooted} if this family consists of monic polynomials with only real roots. If $f$ is the characteristic polynomial of a symmetric matrix with entries in $\R[x]$, it is obviously real rooted. In this article the converse is established, namely that every real rooted bivariate polynomial is the characteristic polynomial of a symmetric matrix over the univariate real polynomial ring.
	As a byproduct we present a purely algebraic proof of the Helton-Vinnikov Theorem which solved the 60 year old Lax conjecture on the existence of definite determinantal representation of ternary hyperbolic forms.
\end{abstract}
\maketitle

\section*{Introduction}

Given a monic polynomial $f\in A[t]$ over a commutative ring $A$ we call a square matrix $M\in \Mat_n A$ a \emph{spectral representation of $f$ over $A$} if $f$ is the characteristic polynomial of $M$, i.e., $f=\det(tI_n-M)$. The main result of this paper is the following
\begin{mthm}\label{thm:main}
	Let $f\in\R[x,t]$ be real rooted, i.e., monic in $t$ and for all $a\in\R$ the univariate polynomial $f(a,t)\in\R[t]$ has only real roots. Then $f$ admits a symmetric spectral representation over $\R[x]$, i.e., there exists $M\in\Sym_n\R[x]$ such that $f=\det(tI_n-M)$. 
\end{mthm}

\subsection*{Symmetric Spectral Representations as Certificates of Real Rootedness}

Given a commutative ring $A$, it is generally a difficult problem to characterize those monic polynomials $f\in A[t]$ that admit a symmetric spectral representation over $A$. As noted above, in the case where $A$ is the polynomial ring $\R[x]$ there is an obvious necessary condition, namely that $f$ is real rooted. This means for every homomorphism $\R[x]\to \R$ the image of $f$ in $\R[t]$ (under coefficientwise application) has only real roots.  The following generalization of this property is shared by all characteristic polynomials of symmetric matrices over any commutative ring $A$: We call $f\in A[t]$ \emph{real rooted over $A$} if $f$ is monic and for all ring homomorphisms from $A$ to any real closed field $R$ the image of $f$ in $R[t]$ has only roots in $R$. In the case $A=\R[x]$ it suffices to check homomorphisms to $\R$ and hence this is indeed a generalization, see Remark~\ref{rem:real_rooted}.

Now it natural to ask about the converse: Which \emph{real rooted} polynomials admit a symmetric spectral representation, or some related, possibly weaker, representation that manifests the real rootedness?

The following characterization of real rooted polynomials over fields is due to Krakowski \citep{Krakowski58}: If $K$ is any field of characteristic different from $2$ then $f\in K[t]$ is real rooted over $K$ if and only if a power of $f$ admits a symmetric spectral representation over $K$. See also \citep{Koulmann01} for a generalization and some lower and upper bounds on the exponent needed.

A useful reformulation of the existence of symmetric spectral representations has been given by Bender \citep{Bender67_rings}, generalizing a result of Latimer and MacDuffee \citep{Latimer_MacDuffee33}, who established a correspondence between equivalence classes of spectral representations of a polynomial $f$ over the ring of integers $\Z$ and ideal classes in $\Z[t]/(f)$. Bender's observation in \citep{Bender67_rings} serves as an inspiration for the present work as it did for Bass, Estes and Guralnick who proved in  \citep{Bass_Estes_Guralnick94} that if $A$ is a Dedekind domain and $f\in A[t]$ real rooted, then $f$ divides the characteristic polynomial of a symmetric matrix over $A$. In other words this means that all roots of $f$ are eigenvalues of a symmetric matrix. Using this result the eigenvalues of adjacency matrices of regular graphs are characterized.

For a slightly smaller class of polynomials, their result can be further extended: A monic polynomial over $A$ is \emph{strictly real rooted} if for any homomorphism $A\to R$ to a real closed field $R$ all roots of the image of $f$ in $R[t]$ lie in $R$ {\it and are simple}. Kummer recently showed in \citep{Kummer16} that for any integral domain $A$ every strictly real rooted polynomial $f\in A[t]$ divides the characteristic polynomial of a symmetric matrix.

The first result towards classification of polynomials that admit symmetric spectral representations without an additional factor  is also due to Bender \citep{Bender68_fields}: If $K$ is a number field and $f\in K[t]$ real rooted over $K$ with an odd degree factor, then $f$ admits a symmetric spectral representation over $K$. It makes essential use of Hasse's local global principle for quadratic forms. A geometric counterpart of this number theoretic theorem holds without restriction: If $K$ is a univariate function field over $\R$ then any real rooted polynomial over $K$ admits a symmetric spectral representation. This follows from Krüskemper's work on scaled trace forms \citep{Krueskemper89}. Ultimately it is a consequence of another local global principle for quadratic forms, due to Witt. See \citep[Lemma~1.5]{Fitzgerald94} and also \citep{Hanselka15} for a more direct proof using an argument by Leep.

The main result of the present paper, Theorem \ref{thm:main}, can be read as a strengthening of two of the aforementioned ones: 
In contrast to the general case of Dedekind domains in \citep{Bass_Estes_Guralnick94}, no additional factors are required. Moreover, it is a denominator free version of the case of the real rational function field: If the coefficients of the polynomial in question are denominator free, then it admits a {\it denominator free} symmetric spectral representation. Using transformations of the form $f\mapsto a^{-d}f(at)$ it is easy to deduce from this the version with denominators. However, our main argument relies less on the theory of quadratic forms but rather on classical theory of divisors on algebraic curves.

The previous results reveal the exceptionality of the case $\R[x]$ over which the class of real rooted polynomials consists \emph{exactly} of the characteristic polynomials of symmetric matrices. In fact, this seems to be essentially the only known nontrivial example of a ring, that is not a field and for which these two classes of polynomials coincide.

\subsection*{Application to Hyperbolic Polynomials}
Closely related to spectral representations are linear determinantal representations of forms in several variables $F\in\R[x_1,\dots,x_{\ell}]$. These are linear pencils of the form
\[
	L=A_1x_1+\dots+A_{\ell}x_{\ell}\quad (A_i\in\Mat_n(\R),\ n=\deg F)
\]
with determinant $F$. We apply our main result to obtain linear symmetric determinantal representations of ternary hyperbolic forms.  A homogeneous polynomial $F\in\R[x_1,\dots,x_{\ell}]$ is called \emph{hyperbolic} with respect to some direction $e\in \R^{\ell}$ if $F(e)>0$ and all real lines in this direction intersect the projective hypersurface defined by $F$ only in real points, i.e., for all $a\in\R^{\ell}$ the univariate polynomial $F(te-a)\in\R[t]$ has only real roots.

We show that a consequence of Theorem~\ref{thm:main} is the following well-known result.

\begin{mthm}[Helton-Vinnikov]\label{thm:heltonvinnikov}
	Let $F\in\R[x,y,z]$ be hyperbolic with respect to $e\in\R^3$. Then there exists a real symmetric matrix pencil $L=Ax+By+Cz$ ($A,B,C\in\Sym_n\R$) such that $L(e)$ is positive definite and $F=\det L$.
\end{mthm}
The original proof can be found in \citep{Helton_Vinnikov07}. It relies on transcendental tools from algebraic geometry such as theta functions on the Jacobian of a Riemann surface. In sharp contrast, our treatment involves only purely algebraic ingredients. 
The statement of Theorem~\ref{thm:heltonvinnikov} has been conjectured by Lax in 1958, see \citep{Lax58}. Its solution also settled a question of Parrilo and Sturmfels in \citep{Parrilo_Sturmfels01} for the characterization of the plane convex semi-algebraic sets that are the feasible sets of semidefinite programming and their minimal descriptions. See \citep{Vinnikov12} for a survey on related problems. Moreover, a slightly weaker version (see Section \ref{sec:hermitian}) has been used in the celebrated proof of the Kadison-Singer Conjecture by Marcus, Spielman and Srivastava in \citep{Marcus_etal13} and by Speyer in \citep{Speyer05} to give another proof of the affirmative answer to Horn's problem on eigenvalues of sums of Hermitian matrices.

\subsection*{Sketch of the Proof}

The proof of Theorem~\ref{thm:main} is an adaptation of the following simple construction with constant coefficients. Suppose $f\in\R[t]$ is a monic polynomial with real simple roots $\la_1,\dots,\la_n\in\R$. We are going to find a real symmetric spectral representation $M\in\Sym_n\R$ of $f$ without computing the roots of $f$. To this end we define the finite-dimensional $\R$-algebra $B:=\R[t]/(f)$ and the vector space endomorphism $\mu$ of $B$ that is given by multiplication by $\overline{t}=t+(f)$. As is easily verified, $f$ is the characteristic polynomial of $\mu$. Moreover, $\mu$ is obviously self-adjoint with respect to the trace form
\begin{align*}
	\ta\colon B\times B&\to\R \\
	(\overline{g},\overline{h})&\mapsto\Tr_{B|\R}(\overline{g}\overline{h})=\sum_igh(\la_i)
\end{align*}
which is positive definite and hence admits an orthonormal basis $\cB$. Now the representing matrix $M$ of $\mu$ with respect to $\cB$ has characteristic polynomial $f$ and is symmetric.

Besides basic field operations this construction only involves taking square roots of positive real numbers in the orthonormalization step. The obvious obstacle to generalizing this construction to coefficient rings $A$ other than $\R$ is the non-existence of an orthonormal basis of the trace form. To overcome it, we follow Bender's approach in \citep{Bender67_rings} by scaling the trace form and replacing the $A$-algebra $B$ by a suitable $B$-ideal.

\medskip

Suppose now $f\in\R[x,t]$ is real rooted and irreducible. We give an outline of the main ideas involved in the construction of a symmetric spectral representation of $f$ as in Theorem~\ref{thm:main}. For simplicity we want to assume that $f$ defines a smooth plane curve. See Remark \ref{rem:smoothness} below on how this can be avoided.

\bigskip
(1)\quad
Consider the extension $B|A$, where $B:=\R[x,t]/(f)$ is the coordinate ring of the curve $\cC$ defined by $f$ and $A:=\R[x]$ the coordinate ring of the $x$-axis, the real affine line. As above, the trace form $\ta$ of $B|A$ is positive semidefinite in all real points. However, the main difficulty in finding an orthonormal basis for $\ta$ is that it is singular in some complex points, namely in those $a\in\C$ where $f(a,t)$ has multiple roots. In other words, these are the ramification points of the projection $\pi$ of $\cC$ onto the $x$-axis. 

\bigskip
(2)\quad
The ramification locus of $\pi$ consists of the points with vertical tangent, i.e., the zeros of
\[
	\de:=\overline{\frac{\partial f}{\partial t}}\in B
\] 
on the curve $\cC$. Rescaling the trace form by $\de^{-1}$ makes it regular everywhere, see Remark~\ref{rem:euler}. On the other hand we lose positivity. This rescaled trace form is totally indefinite in all real points by Rolle's Theorem since the derivative changes sign between two consecutive real roots.

\bigskip
(3)\quad
Next we replace the scaling factor $\de^{-1}$ by a function that assures both definiteness and regularity. For this it must have essentially (up to even order) the same zeros and poles but be positive in all real points.

This is made precise in Lemma~\ref{lem:unimodular_trace_form}: If the $B$-ideal $(c)$ differs from  $\left(\de^{-1}\right)$ by the square of a fractional ideal $I$ then the trace form, rescaled by $c$, restricts to a form on $I$ that is everywhere regular. Moreover, it is positive definite in all real points, if in addition $c$ is a sum of squares. In this case $I$ admits an orthonormal basis $\cB$ with respect to this scaled trace form. Here we use a property that is specific to univariate polynomial rings over fields: Every invertible symmetric matrix is congruent to a constant one, see Theorem~\ref{thm:diagonalization}.

\bigskip
(4)\quad
As in the constant case described above we now choose $M$ to be the representing matrix of $\mu$ with respect to this orthonormal basis $\cB$, where $\mu$ is multiplication by $\overline{t}$ as an $\R[x]$-endomorphism of the ideal $I$. Then $M$ is a symmetric polynomial matrix and its characteristic polynomial is $f$, as desired.

\bigskip
(5)\quad
The main work lies in finding a sum of squares $c$ and a fractional $B$-ideal $I$ as in (3). To this end we first use the real rootedness of $f$ and smoothness of $\cC$ to show that the projection $\pi$ is unramified in all real points, i.e., the curve $\cC$ has no real vertical tangents. This is essentially Corollary~\ref{cor:really_unramified}. Consequently, the ramification points of $\pi$ are nonreal and hence come in complex conjugate pairs. In other words, we can factor the ideal $\De:=(\de^{-1})$ over $\C$ into a product $\overline{J}J$ of a fractional ideal $J$ and its conjugate. Now we use the $2$-divisibility of the class group of $B\otimes\C=\C[x,t]/(f)$, Theorem~\ref{thm:divisible}, to conclude that $J=eE^2$ for some $e\in\C[x,t]/(f)$ and a fractional ideal $E$. Taking $c$ the norm of $e$ and $I$ the norm of $E$ gives us the desired factorization $\De=cI^2$ with $c$ a sum of squares. This is carried out in detail in Corollary~\ref{cor:divisible}.

\begin{rem}\label{rem:smoothness}
	In the proof of Theorem~\ref{thm:main} the smoothness assumption is avoided by working with the normalization of the curve instead, i.e., $\R[x,t]/(f)$ is replaced by its integral closure $B$. The role of the derivative $\partial f/\partial t$ is replaced by the more abstractly defined but technically simpler \emph{codifferent ideal}, the dual module of $B$ with respect to the trace form of $B$ over $\R[x]$. However, if the curve is smooth, then $\R[x,t]/(f)$ is integrally closed. In particular $B$ coincides with $\R[x,t]/(f)$ and is therefore a primitive ring extension of $\R[x]$ and the description of the codifferent becomes more concrete, see Remark~\ref{rem:euler}.
	
	It is not hard to see that in the smooth case every symmetric spectral representation of $f$ arises in the way pointed out above. We can even describe their equivalence classes in terms of pairs $(I,c)$ as in (3), where equivalence of representations is induced by the action of the orthogonal group. For a more precise statement see \citep[Theorem 3.22]{Hanselka15}. If the curve is not smooth, then $\R[x,t]/(f)$ is not integrally closed. The symmetric spectral representations that are produced in the proof of Theorem~\ref{thm:main} are those that extend to homomorphisms from of the integral closure $B$ of $\R[x,t]/(f)$ to $\Sym_n\R[x]$. However, it is not clear which representations of $f$ extend to $B$ in this case.
\end{rem}

\subsection*{Reader's Guide} Section \ref{sec:preliminaries} consists of a collection of definitions, notations, conventions and general facts.
In Section \ref{sec:scaled_trace_forms} we recall a few properties of trace forms and their connection to ramification. Among these is a variant of Bender's result, how scaling a trace form and changing its domain can lead to a unimodular, i.e., everywhere regular form. 
Section \ref{sec:real_rooted} contains some general observations on real rooted polynomials and their interplay with trace forms. From these we deduce that the ramification points described above in the proof's outline are nonreal.
In Section \ref{sec:narrow_class_group} we characterize ideals that admit a factorization as required for the codifferent ideal and which is described in (5) of the outline of the proof.
In Section \ref{sec:proof} we combine the previous sections to give a proof of Theorem \ref{thm:main}.
Section \ref{sec:real_valuations} is concerned with the growth behavior of eigenvalues of symmetric matrices. Applied to polynomial matrices this gives a degree bound for their entries in terms of the coefficients of their characteristic polynomial. We use it to derive the Helton-Vinnikov Theorem~\ref{thm:heltonvinnikov} from our main result.
Section \ref{sec:hermitian} outlines how the proof of Theorem \ref{thm:main} can be simplified to obtain slightly weaker representations, namely complex Hermitian instead of real symmetric ones.

\section{Preliminaries}
\label{sec:preliminaries}
In this section we list some definitions, notations and conventions as well as some of the basic facts that are used throughout the text.
Since our methods are purely algebraic we will not make use of any topological properties of the fields of real and complex numbers. Accordingly $\R$ denotes some arbitrary fixed real closed field and $\C$ the algebraic closure of $\R$. We will refer to the elements of $\R$ and $\C$ as real and complex numbers, respectively.

\subsection*{Notions from Commutative Algebra}
	Let $A$ be a commutative ring, which will always be assumed to have a unit.
	\begin{enumerate}
		\item $\Mat_n A$ and $\Sym_n A$ are the sets of $n\times n$ matrices and symmetric matrices, respectively.
		\item An extension $B|A$ is \emph{finite}, if $B$ is finitely generated as an $A$-module.
		\item $\Spec A$ and $\Sper A$ are the spectrum and real spectrum of $A$, respectively. For the definition of the real spectrum and a general reference on real algebraic geometry see \citep{Marshall08} (also \citep{Bochnak_Coste_Roy_98} or \citep{Prestel_Delzell01}).
		\item For $\p\in\Spec A$ we denote $k(\p):=\Quot (A/\p)$ the residue field of $\p$.
		\item A prime ideal $\p\in\Spec A$ is \emph{real} if $k(\p)$ is formally real, i.e., admits an ordering.
		\item A symmetric bilinear form $\be\colon M\times M\to A$ on an $A$ module $M$ is \emph{unimodular} if $M$ is isomorphic to its own dual via $\be$, i.e., the induced map
			\begin{align*}
				M&\to \Hom_A(M,A)\\
				a&\mapsto \be(a,\cdot)
			\end{align*}
			is an isomorphism.
			\item If $B$ is an $A$-algebra, free of finite rank as an $A$-module, then the \emph{trace form of $B|A$} is the symmetric bilinear form
				\begin{align*}
					\ta_{B|A}\colon B\times B&\to A\\
					(a,b)&\mapsto \Tr_{B|A}(ab)
				\end{align*}
				where for $x\in B$ the trace $\Tr_{B|A}(x)$ is the trace of the $A$-endomorphism of $B$ that is given by multiplication by $x$.
			\item Let $f\in A[t]$ be a monic polynomial and $B:=A[t]/(f)$. The \emph{Hermite matrix} of $f$ is the representing matrix $H$ of $\ta_{B|A}$ with respect to the standard basis $1,\overline{t},\dots,\overline{t}^{n-1}$ of $B$.
	\end{enumerate}
	The importance of the trace form for us lies in the following well-known classical result on real root counting.
	\begin{lem}[Sylvester]\label{lem:signature_trace_form}
		Let $K$ be an ordered field with real closure $R$ and $f\in K[t]$ monic. Then the signature of the trace form of $K[t]/(f)$ over $K$ is the number of distinct roots of $f$ that lie in $R$.
	\end{lem}

	Now let $A$ be a Dedekind domain. For basic theory of Dedekind domains we refer to \citep[Chapter I]{Serre79_book}.
	\begin{enumerate}
		\item By $\cI_A$ we denote the group of nonzero fractional $A$-ideals. It is freely generated by the nonzero elements of $\Spec A$.
		\item The \emph{class group} of $A$, denoted by $\Cl A$, is the quotient of $\cI_A$ modulo the subgroup of principal ideals.
		\item For conceptual reasons we also define the finer \emph{narrow class group} $\Cl_+ A$ to be the quotient of $\cI_A$ modulo the subgroup of those principal ideals that are generated by a sum of squares.
		\item If $\p$ is a nonzero prime ideal of $A$ we denote the $\p$-adic valuation of the field of fractions of $A$ by $v_{\p}$ and for a fractional $A$-ideal $I$ we write
			\[
				v_{\p}(I):=\min \{\, v_{\p}(a) \mid a\in I\,\}
			\]
			for the multiplicity of $\p$ in the prime ideal factorization of $I$.
	\end{enumerate}

\subsection*{Unimodular forms over polynomial rings}
We will make essential use of the following special feature of univariate polynomial rings over fields, generalizing the well known fact that they have only constant units.
\begin{thm}[Harder/Djoković]\label{thm:diagonalization}
	Let $k$ be a field of characteristic different from $2$ and $M$ a free $k[x]$-module of rank $n$.  Then any unimodular bilinear form $\be$ on $M$ admits an orthogonal basis $q_1,\dots,q_n$. Moreover, for every such orthogonal basis we have
	\[
		\be(q_i,q_i)\in k^{\times}.
	\]
\end{thm}
\begin{proof}
	See for example \citep{Djokovic76} or \citep[Theorem~6.3.3]{Scharlau11_book}.
\end{proof}

\section{Scaled Trace Forms and the Codifferent}
\label{sec:scaled_trace_forms}
The trace form of a finite ring extension is in general not unimodular. This is the main obstacle to finding an orthonormal basis in the proof of our Theorem~\ref{thm:main}. Lemma~\ref{lem:unimodular_trace_form} shows how one can overcome this by scaling the trace form appropriately. The relation to ramification can be found in Lemma \ref{lem:different_unramified}. 

\subsection*{The complementary module}
Let $B|A$ be a finite extension of integral domains and assume the extension of their respective fields of fractions $L|K$ is separable. For an $A$-submodule $M$ of $L$ we denote
\[M':=\{\, x\in L \mid \Tr_{L|K}(xM)\subseteq A\,\}\]
the \emph{complementary module} of $M$ and by $\De(B|A):=B'$ the \emph{codifferent} of $B|A$, which is a fractional $B$-ideal.

The following is a variant of Bender's observation in \citep{Bender67_rings}.
\begin{lem}\label{lem:unimodular_trace_form}
	Let $B|A$ be a finite extension of integral domains with separable extension $L|K$ of their respective fields of fractions.
	Further, let $c\in L^{\times}$ and $I$ be an $A$-submodule of $L$ that generates $L$ as a $K$-vector space. We define the scaled trace form
			\begin{align*}
				\be\colon L\times L&\to K\\
				(a,b)&\mapsto\Tr_{L|K}(abc).
			\end{align*}
	\begin{enumerate}[(a)]
		\item The restriction of $\be$ to $I$ is unimodular if and only if $cI=I'$.
		\item If $BI\subseteq I$ then $I'$ coincides with the ideal quotient
			\[
				(\De(B|A):I)=\{\, x\in L \mid xI\subseteq\De(B|A)\,\}.
			\]
		\item If $B$ is a Dedekind domain and $I$ is a fractional $B$-ideal then $\be$ restricts to a unimodular form on $I$ if and only if $cI^2=\De(B|A)$.
	\end{enumerate}
\end{lem}
\begin{proof}
	(a) Since $\ta_{L|K}$ is regular and $I$ generates $L$ as a $K$-vector space, the map
	\begin{align*}
		I'\to \Hom_A(I,A)\\
		x\mapsto \ta_{L|K}(x,\cdot)
	\end{align*}
	is an isomorphism. This means via $\ta_{L|K}$ we can identify the complementary module $I'$ with the dual module of $I$. So via the scaled trace form $\be$ the dual of $I$ becomes the scaled complementary module $c^{-1}I'$. Further, $\be$ is unimodular on $I$ if and only if $I$ coincides with its own dual, i.e., $I=c^{-1}I'$.

	Part (b) follows immediately from the definition and (c) is just a combination of (a) and (b) using the fact that if $B$ is a Dedekind domain then $I$ is invertible and the ideal quotient $(\De(B|A):I)$ can thus be written as $\De(B|A)I^{-1}$.
\end{proof}
\begin{rem}\label{rem:euler}
	The codifferent, the role of the scaling factor in (a) of the previous lemma as well as the relation to vertical tangents become more concrete in the case of primitive ring extensions. For this we use a lemma often attributed to Euler \citep[Lemma~III.6.2]{Serre79_book}: Let $A$ be an integral domain and let $f\in A[t]$ be monic with only simple roots, $f':=\partial f/\partial t$
	its formal derivative and denote $B:=A[t]/(f)$. Then the scaled trace form
	\begin{align*}
		\be\colon B\times B&\to A\\
		(a,b)&\mapsto \Tr_{B|A}\left(\frac{ab}{f'(\overline{t})}\right)
	\end{align*}
	is well defined and unimodular. In particular $\De(B|A)=\left(\frac1{f'(\overline{t})}\right)$. 
\end{rem}

\subsection*{The codifferent encodes ramification}
Roughly speaking, the next lemma states that the support of the codifferent only contains ramified primes. A more precise statement about the ramification index is know as Dedekind's Different Theorem, see \citep[Theorem~III.2.6]{Neukirch99}. A proof of the following can also be found in \citep[Theorem~III.5.1]{Serre79_book}, but since it is short we include it for self-containedness. 
\begin{lem}\label{lem:different_unramified}
	Let $B|A$ be a finite extension of Dedekind domains and let $\p$ be a nonzero prime ideal of $A$ such that $\q|\p$ is unramified and $k(\q)|k(\p)$ is separable for all primes $\q$ of $B$ lying above $\p$. Then none of the latter appears in the prime ideal factorization of $\De(B|A)$, i.e., $v_{\q}(\De(B|A))=0$ for all $\q\in\cI_B$ lying above $\p$.
\end{lem}
\medskip

\begin{proof}
	By considering the localization at $\p$ it suffices to assume that $A$ is a discrete valuation ring with maximal ideal $\p$ and prove that $\De(B|A)=B$.
	
	By Lemma~\ref{lem:unimodular_trace_form} this is equivalent to the trace form $\ta_{B|A}$ being unimodular. Since $A$ is a discrete valuation ring it suffices to show that $\ta_{B|A}$ becomes regular modulo the maximal ideal, i.e., that the trace form
	\[
		\ta_{B|A}\otimes k(\p)=\ta_{B\otimes k(\p)|k(\p)}
	\]
	of the residue ring extension is regular. Let $\p B=\prod_i\q_i$ be the prime ideal decomposition of $\p B$. By assumption the $\q_i$ are pairwise distinct and therefore coprime. This means $B\otimes k(\p)=B/\p B=\prod_i k(\q_i)$. Now we see that the 
	trace form of $B\otimes k(\p)$ over $k(\p)$ is regular, since it is the orthogonal sum of the trace forms of the separable extensions $k(\q_i)|k(\p)$.
\end{proof}

\section{Real Rooted Polynomials and the Trace Form}
\label{sec:real_rooted}
In this section we collect some basic properties of real rooted polynomials. In particular their interplay with trace forms is used to show absence of real ramification, see Corollary \ref{cor:really_unramified}.

\bigskip
Let $A$ be a commutative ring and $f\in A[t]$ monic. Recall that $f$ is \emph{real rooted over $A$}, if for every ring homomorphism $A\to R$ to a real closed field $R$ the image of $f$ in $R[t]$ has only roots in $R$.
For systematic reasons we want to replace homomorphisms into real closed fields by points in the real spectrum $\Sper A$ of $A$. For $P\in\Sper A$ with support $\p\in\Spec A$ denote by $R(P)$ the real closure of the (ordered) residue field $k(\p)$ of $P$ and by $f_P:=f\otimes 1\in A[t]\otimes R(P)=R(P)[t]$ the coefficient wise evaluation of $f$ at $P$. 

\begin{enumerate}
	\item We say $f$ is \emph{real rooted in $P$} if all roots of $f_P$ lie in $R(P)$ and accordingly $f$ is \emph{real rooted in $U\subseteq \Sper A$} if $f$ is real rooted in every point in $U$.
	\item In this sense $f$ is real rooted over $A$ if it is real rooted in $\Sper A$.
\end{enumerate}

From Sylvester's Lemma~\ref{lem:signature_trace_form} we immediately get the following
\begin{cor}\label{cor:signature_trace_form}
	Let $A$ be a commutative ring, $f\in A[t]$ monic, $B=A[t]/(f)$ and $\ta:=\ta_{B|A}$ the trace form of $B|A$. Then $f$ is real rooted in $P\in\Sper A$ if and only if $\ta\otimes_AR(P)$ is positive semidefinite.
\end{cor}

\begin{rem}\label{rem:real_rooted}
	(a)\quad
	From the previous corollary it follows that the set $U\subseteq\Sper A$ of points where $f$ is real rooted consists exactly of those points where all the principal minors of the Hermite matrix of $f$ are nonnegative. In particular, it is a basic closed subset of $\Sper A$ with respect to the Harrison topology.

	\medskip
	(b)\quad For $A=\R[x_1,\dots,x_{\ell}]$ we view $\R^{\ell}$ as a subset of $\Sper A$. Then for $a\in\R^{\ell}$ a polynomial $f\in A[t]$ is real rooted in $a$ if $f_a=f(a,t)\in \R[t]$ has only real roots.

	The set of points $\R^{\ell}$ and the set of orderings $\Sper \R(x_1,\dots,x_{\ell})$ of the rational function field are both dense in $\Sper A$. This follows essentially from Tarski's Transfer Principle \citep[Theorem~2.4.3]{Marshall08} and from the Baer-Krull correspondence \citep[Section~1.5]{Marshall08}, respectively. In particular real rootedness of $f$ in $\R^{\ell}$, $\Sper \R[x_1,\dots,x_{\ell}]$ and $\Sper \R(x_1,\dots,x_{\ell})$ are all equivalent. 
\end{rem}

We make use of the following special local case of this transfer argument, which can be treated completely elementary.
\begin{lem}\label{lem:real_rooted_transfer}
	Let $f\in \R[x,t]$ be real rooted in a neighborhood of the origin of $\R\subseteq \Sper \R[x]$. Then $f$ is real rooted over the field of Laurent series $\R((x))$.
\end{lem}
\begin{proof}
	Let $H\in\Sym_n\R[x]$ be any representing matrix of the trace form of $\R[x,t]/(f)$ over $\R[x]$, e.g. the Hermite matrix of $f$. Using Corollary~\ref{cor:signature_trace_form} we get that $H(a)$ is positive semidefinite for all $a$ in some neighborhood of $0$ and we want to conclude that $H$ is positive semidefinite with respect to both orderings of $\R((x))$. To see this in an elementary way we diagonalize $H$ as a quadratic form over $\R(x)$. Then the resulting diagonal entries are nonnegative rational functions on $(-\ep,\ep)$ for some $\ep\in\R_{>0}$ and thus lie in the preordering generated by $\ep+x$ and $\ep-x$, which is the set of elements of the form $\si_0+\si_1(\ep+x)+\si_2(\ep-x)$, where the $\si_i$ are sums of squares of elements in $\R(x)$. So they are also nonnegative with respect to the two orderings of $\R((x))$ since both make $\ep\pm x$ positive.
\end{proof}

\begin{lem}\label{lem:real_rooted_split}
	Over $\R((x))$ every real rooted polynomial splits into linear factors.
\end{lem}
\begin{proof}
	Any finite field extension of $\R((x))$ either contains $\C$ or is of the form $\R((x^{\frac1e}))$. Both have nonreal embeddings into the algebraic closure of $\R((x))$ unless $e=1$. That means if $f\in\R((x))[t]$ is real rooted and irreducible over $\R((x))$ then it must be of degree one. 
\end{proof}

As a consequence we get the absence of real ramification that we need for the factorization of the codifferent in the proof of our main result. Similar results can be found in \citep[Corollary to Lemma 4.1]{Dubrovin85eng}, \citep[Theorem~6.2]{Bass_Estes_Guralnick94} and a higher dimensional generalization in \citep[Theorem~2.19]{Kummer_Shamovich15}.
\begin{cor}\label{cor:really_unramified}
	Let $f\in\R[x,t]$ be irreducible and denote $K:=\R(x)$ and $L:=K[t]/(f)$. If $f$ is real rooted in a neighborhood of $a\in\R$ then the $(x-a)$-adic valuation of $K$ is unramified in $L|K$.
\end{cor}
\begin{proof}
	Let $f$ be real rooted in a neighborhood of $a$ which we can assume to be the origin. By Lemma~\ref{lem:real_rooted_transfer} it is also real rooted over $\R((x))$. Let $v$ be the $x$-adic valuation of $\R(x)$ and $w$ an extension of $v$ to $L$. Then the completion $L_w$ is a factor in $L\otimes\R((x))$ which must be of degree one by Lemma~\ref{lem:real_rooted_split}, i.e., $L_w=\R((x))$. In particular, $w|v$ is unramified.
\end{proof}

\section{Squares in the Narrow Class Group}
\label{sec:narrow_class_group}
Recall that the narrow class group of a Dedekind  domain $A$ is the ideal group $\cI_A$ modulo the subgroup of sum of squares principal ideals. In Corollary~\ref{cor:divisible} we characterize the squares in the narrow class group of a smooth affine curve over $\R$, which is the essential step in finding positive definite unimodular scaled trace forms in the proof of Theorem \ref{thm:main}. This characterization is a consequence of the $2$-divisibility of the class group of a smooth affine curve over $\C$.

\begin{thm}\label{thm:divisible}
	If $A$ is a Dedekind domain that is a finitely generated $\C$-algebra, then its class group is divisible.
\end{thm}
\begin{proof}
	The class group $\Cl A$ is a quotient of the degree zero part $\Cl^0 K$ of the divisor class group of the univariate function field $K=\Quot A$ over $\C$.

	A direct proof of the divisibility of $\Cl^0 K$ is due to Frey \citep{Frey79} and holds even in positive characteristic. More geometric arguments rely on the Jacobian of the smooth curve corresponding to the function field $K$. See, e.g. \citep[Section 2.2]{Griffith_Harris78} for a classical analytic treatment or \citep[p. 42]{Mumford70} for an approach using Weil's algebraic generalization.
\end{proof}

\begin{cor}\label{cor:divisible}
	Let $A$ be a Dedekind domain that is a finitely generated $\R$-algebra and $J\in\cI_A$ a fractional $A$-ideal. Then the class of $J$ is a square in the narrow class group $\Cl_+ A$ if and only if all real prime ideals appear in $J$ with even order. In other words there exists $I\in\cI_A$ and a sum of squares $c\in\Quot A$ such that $J=cI^2$ if and only if $2|v_{\p}(J)$ for every nonzero real $\p\in\Spec A$.
\end{cor}
\begin{proof}
	The last condition is clearly necessary, since in general the value of a sum of squares under any real valuation is divisible by $2$, see \citep[Exercise 1.4.10]{Prestel_Delzell01}. For the converse let $v_{\p}(J)$ be even for every real prime $\p$. Multiplying $J$ by an appropriate product of even powers of real prime ideals we can even assume that $J$ is a product of nonreal prime ideals and their inverses. It thus suffices to show that the class of every nonreal prime ideal is a square in $\Cl_+ A$.

	If $\C\subseteq A$ then $\Cl_+ A=\Cl A$ and the claim follows directly from Theorem~\ref{thm:divisible}. Assume now that $-1$ is not a square in $A$ and hence the finitely generated $\C$-algebra $B:=A\otimes_{\R}\C$ is again a Dedekind domain the class group of which is divisible, again by Theorem~\ref{thm:divisible}.

	Now let $\p\in\cI_A$ be a nonreal prime ideal. We want to show that its class in $\Cl_+ A$ is a square. The norm of an element of $B$ is a sum of two squares in $A$. Therefore, the ideal norm map $\Norm_{B|A}$ %(see e.g. \citep[Section~I.5]{Serre79_book})
	induces a homomorphism $\Cl B\to \Cl_+ A$. Using the $2$-divisibility of $\Cl B$ it thus suffices to show that $\p$ is the norm of an ideal in $B$. Since $\p$ is nonreal we have $k(\p)=\C$. Hence for $\q\in\cI_B$ lying above $\p$ the extension $k(\q)$ of $k(\p)$ is trivial, so the residue degree $f_{\q|\p}=[k(\q):k(\p)]$ is $1$. So we get $\Norm_{B|A}(\q)=\p^{f_{\q|\p}}=\p$, as desired.

	More concretely this means that $\p$ corresponds to a pair of conjugate points on the affine curve $\Spec B$. Hence $\p$ factors over $\C$ into a product of two conjugate prime ideals. The norm of each of these two factors equals $\p$.
\end{proof}

\section{Proof of the Main Theorem}
\label{sec:proof}
Now we have collected all the necessary ingredients to prove Theorem \ref{thm:main}. Let $f\in\R[x,t]$ be real rooted, i.e., $f$ is monic in $t$ and $f(a,t)$ has only real roots for all $a\in\R$. To prove that $f$ is the characteristic polynomial of a symmetric matrix over $\R[x]$ we may assume that $f$ is irreducible. Otherwise we find a symmetric spectral representation of each of its irreducible factors and compose them to a block diagonal matrix which then gives a symmetric spectral representation of $f$.

We fix the following notation:
\begin{itemize}
	\item $n=\deg_t f$,
	\item $A=\R[x]$ the coordinate ring of the real affine line $\A^1_{\R}$,
	\item $K=\R(x)$ its function field,
	\item $L=K[t]/(f)$ the function field of the plane affine curve $\cC$ defined by $f$,
	\item $B$ the integral closure of $A$ in $L$, i.e., the coordinate ring of the normalization $\widetilde{\cC}$ of $\cC$,
		\begin{center}$
			% old tikzcd to work with arxiv?
			\begin{tikzcd}
				\widetilde{\cC}\arrow{d}	& B\arrow{r}{\subseteq}\arrow[-]{d}	& L\arrow[-]{d}{n}\\
				\A^1_{\R}	& A\arrow{r}{\subseteq}			& K
			\end{tikzcd}$
			%$\begin{tikzcd}
			%	\widetilde{\cC}\ar[d]	& B\ar[r,"\subseteq"]\ar[d, dash]	& L\ar[d,"n",dash]\\
			%	\A^1_{\R}	& A\ar[r,"\subseteq"]			& K
			%\end{tikzcd}$
		\end{center}
	\item $\ta=\ta_{L|K}$ the trace form of $L|K$.
	\item $\De=\De(B|A)$ the codifferent of $B|A$.
\end{itemize}

We combine our preparatory work as outlined in the introduction. Since $f$ is real rooted in every point $a\in\R$, the extension $B|A$ is unramified in all real primes of $A$ by Corollary~\ref{cor:really_unramified}. Therefore, $v_{\q}(\De)=0$ for all real primes $\q\in\cI_B$ by Lemma~\ref{lem:different_unramified}. Using Corollary~\ref{cor:divisible} it now follows that the class of $\De$ in the narrow class group $\Cl_+ B$ is a square, i.e., there exists a sum of squares $c\in L^{\times}$ and a fractional ideal $I\in \cI_B$ such that $cI^2=\De$. By Lemma~\ref{lem:unimodular_trace_form}(c) the scaled trace form
\begin{align*}
	\be\colon I\times I&\to A\\
	(a,b)&\mapsto\Tr_{L|K}(abc)
\end{align*}
is well-defined and unimodular. Since $A$ is a principal ideal domain and $I$ is finitely generated and torsion free as an $A$-module, it is already free. Now by Theorem~\ref{thm:diagonalization} of Harder and Djoković we can orthogonalize it with nonzero real numbers on the diagonal. These must be positive as follows easily from Sylvester's Lemma~\ref{lem:signature_trace_form} since $c$ is a sum of squares. This means $(I,\be)$ admits an orthonormal basis which we denote by $\cB$.

Denote $\mu$ multiplication by $\overline{t}$, viewed as an endomorphism of the $K$-vector space $L$. Its characteristic polynomial is $f$. Since any $A$-basis of $I$ is also a $K$-basis of $L$, the restriction of $\mu$ to $I$ has characteristic polynomial $f$ as well.

Since $\mu$ is obviously self-adjoint with respect to $\be$, its representing matrix $M\in\Mat_n A$ with respect to the orthonormal basis $\cB$ of $I$ is symmetric, hence $M$ is a symmetric spectral representation of $f$ over $A$, as desired.\qed

\section{Symmetric Matrices and Real Valuations}
\label{sec:real_valuations}
The size of the entries of a symmetric matrix over $\R$ can be bounded in terms of its eigenvalues and hence in terms of the coefficients of its characteristic polynomial. We give a valuation theoretic analogue of this observation. Applied to the degree valuation this shows that the Helton-Vinnikov Theorem~\ref{thm:heltonvinnikov} follows from Theorem~\ref{thm:main}.

\bigskip
In the following let $v$ be a real valuation on $K$, i.e., the residue field $\overline{K}$ is formally real. Let $M\in \Mat_n K$. Denote $v(M)$ the minimal value of the entries of $M$. We obtain an obvious lower bound on the values of the coefficients of its characteristic polynomial $f=\det(tI_n-M)=\sum_ia_it^i\in K[t]$ since each $a_i$ is homogeneous of degree $n-i$ in the entries of $M$. In particular we have
\[
	v(a_i)\geq (n-i)v(M).
\]
If the matrix is symmetric then this bound is sharp, i.e., we have equality for at least one $i$:
\begin{prop}\label{prop:value}
	Let $M\in\Sym_n K$ be nonzero and $f=\det(tI_n-M)=\sum_ia_it^i\in K[t]$ $(a_i\in K)$. Then
	\[
		v(M)=\min_{0\leq i<n} \frac{v(a_i)}{n-i}.
	\]
	In particular the right hand side lies in the value group of $v$.
\end{prop}
\begin{proof}
	Let $a\in K^{\times}$ be an entry of $M$ with minimal value, i.e., $v(a)=v(M)$. We rescale $M$ and $f$ so that both lie in the valuation ring of $v$. So we define
	\[
		M_0:=a^{-1}M
	\]
	and 
	\[
		f_0:=\det(tI_n-M_0)=a^{-n}\det(atI_n-M)=a^{-n}f(at)=\sum_i\frac{a_i}{a^{n-i}}t^i.
	\]
	Since the residue field $\overline{K}$ is formally real and $\overline{M_0}\in \Mat_n\overline{K}$ is symmetric and nonzero it cannot be nilpotent. By the Cayley-Hamilton Theorem at least one coefficient of its characteristic polynomial other than the leading one must be nonzero. So there exists $i<n$ such that $\overline{\frac{a_i}{a^{n-i}}}$ is nonzero and hence 
	\[
		v(a_i)=v(a^{n-i})=(n-i)v(M)
	\]
	as claimed.
\end{proof}

Applying this to the case where $v$ is the degree valuation on $K=\R(x)$, i.e., $v=-\deg$, we immediately obtain the following
\begin{cor}\label{cor:degree}
	Let $M\in\Sym_n\R[x]$ and $f=\det(tI_n-M)\in\R[x,t]$ its characteristic polynomial. If the total degree of $f$ is $n$, then $M$ is linear, i.e., its entries have at most degree one.
\end{cor}

Using this it becomes easy to derive the Helton-Vinnikov Theorem from our main result.

\medskip
\begin{proof}[of Theorem \ref{thm:heltonvinnikov}]
	After rescaling $F$ and $e$ and applying a linear change of variables we can assume that $F(e)=1$ and $e=(0,0,1)$. Then the dehomogenization $f:=F(x,1,t)\in\R[x,t]$ is real rooted and thus admits a symmetric spectral representation $M\in\Sym_n\R[x]$ by Theorem~\ref{thm:main}. The condition that $f$ is of total degree $n$ forces the entries of $M$ to be linear, by Corollary~\ref{cor:degree}. This means $M$ is of the form $M_1x+M_0$ for some $M_0,M_1\in\Sym_n\R$. Homogenizing again we see that $F$ is the determinant of the real symmetric pencil $L:=I_nz-M_0y-M_1x$. Moreover, $L(e)=I_n$ is positive definite. 
\end{proof}

\section{Hermitian Spectral Representations}
\label{sec:hermitian}
Finally, we want to sketch how the above procedure can be simplified to produce complex Hermitian instead of real symmetric spectral representations of real rooted polynomials. As usual a matrix $M\in\Mat_n\C[x]$ is \emph{Hermitian} if it equals its conjugate transpose, where conjugation refers to coefficient wise complex conjugation of the entries. 
The process of producing Hermitian representations becomes considerably more elementary, since it does not depend on Theorem \ref{thm:divisible}, the divisibility of the class group.

By the same argument provided in the previous section as well as the appropriate reformulation of Proposition \ref{prop:value} this weaker result can be used to prove existence of definite linear Hermitian determinantal representations of hyperbolic polynomials. This result has been obtained previously by Dubrovin \citep{Dubrovin85eng} and Vinnikov \citep{Vinnikov93}. Further elementary proofs can be found in \citep{Plaumann_Vinzant13} and \citep{Drexel3_Vinnikov16}.

To produce Hermitian representations we replace the symmetric bilinear trace form of $L$ over $K$ in the proof of Theorem \ref{thm:main} by the Hermitian trace form of $\widetilde{L}:=L\otimes_{\R} \C$ over $\widetilde{K}:=\C(x)$ which is given by
\begin{align*}
	\widetilde{\ta}\colon\widetilde{L}\times \widetilde{L}&\to \widetilde{K}\\
	(a,b)&\mapsto \Tr_{\widetilde{L}|\widetilde{K}}(a^*b)
\end{align*}
where $*$ denotes the induced complex conjugation on $\widetilde{L}$. The crucial difference now is the required factorization of the codifferent $\widetilde{\De}$ of $\widetilde{B}:=B\otimes\C$ over $\widetilde{A}:=\C[x]$. Namely $\widetilde{\De}$ is already a Hermitian square, i.e., there exists a fractional $\widetilde{B}$-ideal $I$ such that $I^*I=\widetilde{\De}$. Then $\widetilde{\ta}$ restricts to a unimodular positive definite Hermitian form on $I$. Using the more general version of Theorem \ref{thm:diagonalization} found in \citep{Djokovic76} it therefore admits an orthonormal basis. Now we can proceed as before to get a Hermitian spectral representation of $f$. A more detailed explanation can be found in \citep{Hanselka15}.

\bibliographystyle{alpha}
\bibliography{lit}
\end{document}